\documentclass[12pt]{amsart}
\usepackage{a4wide}
\usepackage{amssymb}
\usepackage{amsthm}
\usepackage{amsmath}
\usepackage{amscd}
\usepackage{verbatim}
\usepackage{color}
\usepackage[mathscr]{eucal}
\textheight8.8in \textwidth6.55in \numberwithin{equation}{section}

\newcommand{\sm}[4]{\left(\begin{smallmatrix}#1&#2\\ #3&#4
\end{smallmatrix} \right)}
\newtheorem{theorem}{Theorem}
\newtheorem{lemma}[theorem]{Lemma}

\newtheorem{proposition}[theorem]{Proposition}
\newtheorem{definition}[theorem]{Definition}

\theoremstyle{remark}

\newtheorem*{remark}{Remark}

\numberwithin{theorem}{section} \numberwithin{equation}{section}

\newcommand{\R}{\mathbb{R}}
\newcommand{\C}{\mathbb{C}}

\newcommand{\Z}{\mathbb{Z}}

\newcommand{\sgn}{\operatorname{sgn}}

\newcommand{\sell}{s\ell}

\begin{document}
\title[Almost harmonic Maass forms]{Almost harmonic Maass forms and Kac-Wakimoto
characters}

\author{Kathrin Bringmann}
\address{Mathematical Institute, University of Cologne, 50931 Cologne,
Germany} \email{kbringma@math.uni-koeln.de}
\author{Amanda Folsom}
\address{Yale University, Mathematics Department, 10 Hillhouse Avenue, P.O. Box 208283
New Haven, Connecticut, USA 06520-8283} \email{amanda.folsom@yale.edu} \subjclass[2010]{11F22, 11F37, 17B67, 11F50}
\begin{abstract}
We resolve a question
of Kac, and explain the automorphic properties of characters due to Kac-Wakimoto pertaining to $s\ell(m|n)^{\wedge}$ highest weight modules, for $n\geq 1$. We prove that
the Kac-Wakimoto characters are essentially holomorphic parts of certain generalizations of  harmonic weak Maass forms which we call ``almost harmonic Maass forms". Using a new approach, this generalizes prior work of the first author and Ono, and the authors, both of which treat only the case $n=1$.  We also provide
an explicit asymptotic expansion for the characters.
\end{abstract}
\maketitle
\section{Introduction and Statement of Results}\label{Intro}
The theory of harmonic weak Maass forms has evolved substantially over the last decade.  Loosely speaking,  harmonic weak
Maass forms, as originally defined by Bruinier and Funke \cite{BFunk}, are non-holomorphic relatives to classical modular forms,
which in addition to satisfying suitable modular transformations, are annihilated by a weight $\kappa$ hyperbolic Laplace operator (defined in \eqref{Laplace}),
and are allowed relaxed cuspidal growth conditions.  Works by the authors,  Bruinier, Ono, Zagier,
Zwegers, and many others, further develop the theory of harmonic weak Maass forms, reaching many diverse
areas of research (see for example \cite{OnoCDM} or \cite{ZagierBourbaki} for a more detailed account).  One particular area in which
harmonic weak Maass forms have recently been shown to play integral roles is the representation theory of affine Lie
superalgebras.  It is now well known that classical modular forms often appear as characters for certain groups, or affine Lie
algebras. For example, one of the most beautiful such results is given by ``Monstrous moonshine". Conway and Norton in 1979 made
a surprising conjecture relating the Fourier coefficients of the modular $j$-function
$$
j(\tau) = \frac{1}{q} + 744 + 196884q + 21493760q^2 + ...
$$
($q:=e^{2\pi i \tau}, \tau \in \mathbb H$), to dimensions of irreducible representations of the Monster group \cite{CN}, and
Borcherds ultimately gave a proof in 1992 \cite{B}.  Prior to this, important work of Kac \cite{K} established the so-called
Kac-Weyl character formula and denominator identity, first relating infinite-dimensional Lie algebras and modular forms.

More
recently, Kac and Wakimoto \cite{KW2} found a specialized character formula pertaining to the affine Lie superalgebra
$s\ell(m,1)^\wedge$ for ${tr}_{L_{m, 1}(\Lambda_{(\ell)})}q^{L_0}$ ($m\geq 2$ an integer), where $L_{m, 1}(\Lambda_{(\ell)})$ is the irreducible
$\sell(m,1)^{\wedge}$ module with highest weight $\Lambda_{(\ell)}$, and $L_0$ is the ``energy operator".  In \cite{BOKac}, the
first author and Ono answered a question of Kac regarding the ``modularity" of the characters
${tr}_{L_{m, 1}(\Lambda_{(\ell)})}q^{L_0}$, and proved that they are not modular forms, but are essentially holomorphic parts of harmonic weak Maass forms.
In a subsequent work \cite{F}, the second author related these characters to universal
mock theta functions.  Moreover, in \cite{BF}, we were able to exploit the ``mock-modularity" of the Kac-Wakimoto
$s\ell(m,1)^\wedge$-module characters to exhibit their detailed asymptotic behaviors.

In this paper, we fully resolve Kac's question, and explain the
automorphic properties for arbitrary $n \in  \mathbb N$
of the Kac-Wakimoto
characters ${tr}_{L_{m, n}(\Lambda_{(\ell)})}q^{L_0}$  for irreducible highest weight  $s\ell(m,n)^\wedge$
modules. We
also provide an explicit asymptotic expansion for these characters. We
point out that there is no reason to expect genuine modular forms, given
the case of \cite{BOKac}
 in the case n=1.  We prove that in fact, the Kac-Wakimoto characters are essentially holomorphic parts of new automorphic objects (defined in \S \ref{autosec}) which we call ``almost harmonic Maass forms".
We point out that unlike the case of $n=1$, one does not have the luxury of
beginning with a multivariable Appell-Lerch sum expression for ${tr}_{L_{m,n}(\Lambda_{(\ell)})}q^{L_0}$ for arbitrary $n$, as was the case in \cite{BOKac} when $n=1$. Instead, we consider the generating function as given by Kac-Wakimoto for the specialized characters $\textnormal{ch}F_\ell$
\begin{align}\label{KWgenf}
\textnormal{ch}F := \sum_{\ell \in \mathbb Z}  \textnormal{ch}F_\ell \zeta^\ell
 =e^{\Lambda_0}\prod_{k=1}^\infty
\frac{\prod_{i=1}^m\left(1+\zeta w_i q^{k-\frac{1}{2}}\right)\left(1+\zeta^{-1}w_i^{-1}q^{k-\frac{1}{2}}\right)}{\prod_{j=1}^{n}\left(1-\zeta
w_{m+j}
q^{k-\frac{1}{2}}\right)
\left(1-\zeta^{-1}w_{m+j}^{-1}q^{k-\frac{1}{2}}\right)},
\end{align}
 where $\textnormal{ch}F = \textnormal{ch}F(\zeta,\tau)$, $\textnormal{ch}F_\ell = \textnormal{ch}F_{\ell}(m,n;\tau)$, and $e^{\Lambda_0}$ and the $w_s$ are certain operators (see (3.15) of \cite{KW2}).  As the coefficient functions $\textnormal{ch}F_\ell$ depend upon the range in which $\zeta$ is taken, we choose a specific range in \S \ref{KWMod}.  By moving to a different range, ``wall crossing" occurs \cite{DMZ}.

To simplify the situation we set (considered as formal variables) $ e^{\Lambda_0}$ and all $w_s=1,$ (also guaranteeing convergence of (\ref{KWgenf})).  Once we have established the automorphic properties of the characters $\textnormal{ch}F_\ell$, it is easy to deduce the automorphic properties of the characters $\textnormal{tr}_{L_{m,n}(\Lambda_{(\ell)})}q^{L_0}$ using the relationship (see \cite{KW2} and \S \ref{AB})
  \begin{align}\label{chtrgenrel}
  \textnormal{tr}_{L_{m, n}(\Lambda_{(\ell)})}q^{L_0} = \textnormal{ch}F_\ell \cdot \prod_{k\geq 1}\left(1-q^k\right),
  \end{align}
  and the fact that $q^{\frac{1}{24}}\prod_{k\geq 1} (1-q^k) = \eta(\tau)$ is a well known modular form of weight $1/2$.
    It is not difficult to show that, upon suitable specialization of variables, the generating function $\textnormal{ch}F$ in (\ref{KWgenf}) is essentially a Jacobi form (see \S  \ref{autosec}-\S \ref{KWMod}).  Ideally then, one might proceed by considering the theory of (holomorphic) Jacobi forms after Eichler and Zagier \cite{EZ}, however the form in question (\ref{KWgenf}) is not holomorphic, but is meromorphic, and moreover, its poles are of arbitrarily high order $n$ (where $n\geq 1$ is a fixed integer).
    At present, the there is no satisfactory theory of meromorphic Jacobi forms.  Zwegers \cite{ZwegersPhD} and Dabholkar-Murthy-Zagier \cite{DMZ} take two different approaches to considering meromorphic Jacobi forms with poles of order $n=1$ and $n=2$.  Here, we extend the approach of Dabholkar-Murthy-Zagier for $n=1$ and $n=2$
    and consider meromorphic Jacobi forms with poles of arbitrary order $n\geq 1$ in the course of our treatment of the more general Kac-Wakimoto characters for $s\ell(m,n)^\wedge$-modules.

In doing so, we encounter new automorphic objects (that do not appear in \cite{DMZ} or \cite{ZwegersPhD} e.g.), which we call \emph{almost harmonic Maass forms of depth $r$}, and define more precisely in \S 2.    Loosely speaking,  such functions may be viewed as sums of harmonic weak Maass forms under iterates of the raising operator (themselves therefore non-harmonic weak Maass forms), multiplied by almost holomorphic modular forms.
 We point out that in particular harmonic weak Maass forms and almost holomorphic modular forms are almost harmonic Maass forms of depth $r=1$.
 We call the associated holomorphic parts \emph{almost  mock modular forms}.
It turns out that the Kac Wakimoto characters are almost mock modular forms.   For simplicity and ease of exposition, we consider here even
positive integers $m$ and $n$ satisfying $m>n\geq 2$.
As is usual  in the subject, all modularity statements are made up to $q$-powers and a finite number of additional terms that may appear in the $q$-expansion.
\begin{theorem}\label{Mainthm}
Assume the restrictions above.
 The Kac Wakimoto characters $\textnormal{tr}_{L_{m,n}(\Lambda_{(\ell)})}q^{L_0}$ are almost mock modular forms  of weight $0$ and depth $\frac{n}{2}$.
\end{theorem}
 \begin{remark}
 In particular, when $n=2$, the proof of Theorem \ref{Mainthm} shows that the Kac-Wakimoto characters are mixed mock modular forms, i.e., products of mock modular forms multiplied by modular forms. (See also \S \ref{examplesec}.)
 \end{remark}

 We will show more precisely in Theorem \ref{decompth2} that the Kac-Wakimoto characters $\textnormal{ch} F_\ell$ appear as coefficients in theta-type decompositions of nonholomorphic Jacobi forms, reminiscent of the holomorphic situation as developed by Eichler-Zagier \cite{EZ} (see \S \ref{autosec}).   We also provide some specific examples in \S \ref{examplesec} below.

We next consider the asymptotic behavior of the Kac-Wakimoto characters, generalizing results in
 \cite{BF}
 and \cite{KW2}.
 More precisely, in the case $n=1$ pertaining to the affine Lie superalgebras $s\ell(m,1)^\wedge$, Kac and Wakimoto considered the specialized characters (see \S 4 of \cite{KW2}) $$\textnormal{tr}_{L_{m, 1}(\Lambda_{(\ell)})}q^{L_0} =  \textnormal{ch} F_\ell\cdot \prod_{k\geq 1} \left(1-q^k\right).$$ Using series manipulations and modular transformations of theta functions, for $\tau = it, t\in \mathbb R^+$, Kac and Wakimoto established the main term in the asymptotic expansion of the specialized characters $\textnormal{tr}_{L_{m, 1}(\Lambda_{(\ell)})}q^{L_0}$ as $t\to 0^+$:
\begin{align}\label{KWasymth1}\textnormal{tr}_{L_{m, 1}(\Lambda_{(\ell)})}q^{L_0}
\sim \frac{\sqrt{t}}{2} e^{\frac{\pi(m+1)}{12 t}}.
\end{align}
In \cite{BF}, using different methods than those in \cite{KW2}, we established an asymptotic expansion with an arbitrarily large number of terms beyond the main term given in (\ref{KWasymth1}).  Here, we generalize the results in \cite{BF} and (\ref{KWasymth1}), and establish the following asymptotic expansion for the specialized characters $\textnormal{tr}_{L_{m, n}(\Lambda_{(\ell)})}q^{L_0}$ pertaining to the affine Lie superalgebras $s\ell(m,n)^\wedge$ for any $n\geq 1$.
\begin{theorem}\label{Asythm} Let $\ell \in \mathbb Z$ and $m >n \geq 1 $.
Then for $\tau = it$, $ t\in \mathbb R^+$,   for any
$N\in\mathbb N_0$, as $t\to 0^+$ we have
\[
{tr}_{L_{m, n}(\Lambda_{(\ell)})}q^{L_0}=e^{\frac{\pi t}{12}(n-m+1)+\frac{\pi}{12 t}(m+2n-1)}\frac{\sqrt{t}}{2^n}
\left(\sum_{r=0}^Na_r(m, n, \ell)\frac{t^r}{ r!}+O\big(t^{N+1}\big)\right),
\]
where
\[
a_r(m, n, \ell):=(-\pi)^r\sum_{{k}=0}^r\left(\begin{matrix} r\\ {k}\end{matrix}\right)(m-n)^{ k}(2 i\ell)^{r-{ k}}\mathscr E_{{ k}+r,n}.
\]
Here the numbers $\mathscr E_{k+r,n}$ are ``higher" Euler numbers defined by the recurrence (\ref{recurEs}) and the initial conditions   (\ref{E1defint}) and (\ref{E2defint}).
\end{theorem}
The remainder of the paper is organized as follows.  In \S \ref{autosec} we define the automorphic objects that emerge in our treatment here, and discuss their transformation properties.  In \S \ref{KWMod} we prove Theorem \ref{Mainthm}, and in \S \ref{examplesec} we provide explicit examples in the cases $n=2$ and $n=4$. In \S \ref{AB} we prove Theorem \ref{Asythm}.

\section*{Acknowledgements}
The authors thank Don Zagier for fruitful conversation. Moreover they thank Karl Mahlburg, Rene Olivetto, and Martin Raum for their comments on an earlier
version of the paper.
The research of the first author was supported by the Alfried Krupp Prize for young University Teachers of the Krupp Foundation.
The second author is grateful for the support of National Science Foundation grant DMS-1049553.

\section{Automorphic Forms}\label{autosec}
In this section, we define relevant automorphic objects, and give their transformation properties.\subsection{Weak Maass forms}
Assume that $\kappa\in\frac12 \Z,$ and $\Gamma$ is a congruence subgroup of either $\text{SL}_2(\Z)$ or $\Gamma_0 (4),$ depending on whether or not $\kappa\in\Z$.
The weight $\kappa$ {\it slash operator}, defined for a matrix $\gamma = \left(\begin{smallmatrix}
a & b \\ c & d
\end{smallmatrix}\right) \in\Gamma$ and any function $f:\mathbb{H} \rightarrow \C,$  is given by
$$
f\big|_\kappa \gamma (\tau) := j(\gamma, \tau)^{-2\kappa} f\left( \frac{a\tau +b}{c\tau +d} \right)
$$
where
$$
j(\gamma, \tau) := \begin{cases}
   \sqrt{c\tau +d} & \text{if } \kappa\in\Z, \\
   \left( \frac{c}d \right)  \varepsilon_d^{-1} \sqrt{c\tau +d}    & \text{if } \kappa\in\frac12 \Z \backslash \Z .
  \end{cases}
$$
Here
$$
\varepsilon_d:=\begin{cases}
   1 & \text{if } d\equiv 1 \pmod{4}, \\
   i       & \text{if } d\equiv 3 \pmod{4}.
  \end{cases}
$$
Moreover we require the \textit{weight $\kappa$ hyperbolic Laplace operator} ($\tau = u +iv$)
\begin{equation}\label{Laplace}
\Delta_\kappa := -v^2 \left(\frac{\partial ^2}{\partial u^2} + \frac{\partial ^2}{\partial v^2} \right) + i\kappa v \left( \frac{\partial}{\partial u} + i \frac{\partial }{\partial v} \right).
\end{equation}

\begin{definition}
Let $\kappa\in\frac12 \Z$, $N$ a positive integer, $\chi$ a Dirichlet character $\text{(mod $N$)}$, $s\in\C$. A \it{Maass form of weight $\kappa$ for $\Gamma$ with Nebentypus character $\chi$ and Laplace eigenvalue $s$} is a smooth function $f:\mathbb{H}\rightarrow \C$ satisfying
\begin{enumerate}
\item For all $\gamma=\left(\begin{smallmatrix}
a & b \\ c & d
\end{smallmatrix}\right)\in\Gamma$ and all $\tau\in\mathbb{H}$, we have
$$
f\big|_\kappa \gamma (\tau ) = \chi (d)f(\tau ).
$$
\item We have that
$$
\Delta_\kappa f = sf.
$$
\item The function $f$ has at most linear exponential growth at all cusps.
\end{enumerate}
\end{definition}
 Of particular interest are {\it harmonic weak Maass forms}, which are Maass forms with eigenvalue $s=0$.
 Harmonic weak Maass forms appear in the work of Bruinier-Funke on theta liftings \cite{BFunk}, for example.  There, one may find a more detailed treatment, and a more precise description of condition (3) in the definition above.    By considering their Fourier expansions, it is known that harmonic weak Maass forms naturally decompose into two summands, a
\emph{holomorphic part} and a \emph{non-holomorphic part} .
For example, it is known \cite{HZ, ZagClass} that the generating function for Hurwitz class numbers $H(n)$ of binary quadratic forms of discriminant $-n$ is (essentially) the holomorphic part of the following weak Maass form of weight $3/2$ and level $4$, the so-called ``Zagier-Eisenstein series":
\begin{align}
\label{ZH} \mathcal{H}(\tau)&:=-\frac{1}{12} + \mathop{\sum_{n\geq 1}}_{n\equiv 0,3 \!\!\pmod 4} H(n)q^n + \frac{1+i}{16\pi} \int_{-\overline{\tau}}^{i\infty} \frac{\Theta(w)}{(w+\tau)^{\frac{3}{2}}}dw,
\end{align}
where $\Theta(\tau) := \sum_n q^{n^2}.$ The holomorphic parts of weak Maass forms are
called \emph{mock modular forms} \cite{ZagierBourbaki}.
We also make use of the fact that harmonic weak Maass forms $G$ of weight $\kappa$ are mapped to ordinary modular forms of weight $2-\kappa$ by the differential operator
$$\xi_k := 2i v^\kappa \cdot \overline{\frac{\partial}{\partial \overline{\tau}}}.$$  If we denote the holomorphic part of $G$ by $G^+$, the image $\xi_\kappa(G)$ is called the \emph{shadow} of the mock modular form $G^+$, and is modular of weight $2-\kappa$. (See \cite[\S 7]{OnoCDM} and \cite{ZagierBourbaki}, for example.) For example, in (\ref{ZH}), the shadow of the mock modular generating function for Hurwitz class numbers is the form $\Theta(\tau)$.

\subsection{Holomorphic modular forms and almost modular forms}
A special interesting subclass of harmonic weak Maass forms that we consider here are the  \textit{weakly holomorphic modular forms}, i.e., those harmonic weak
Maass forms that are holomorphic on the upper half plane but may have poles in the cusps of $\Gamma$.
If the modular forms are also holomorphic in the cusps, we speak of \textit{holomorphic modular forms}.
A  special ordinary modular form required in this paper is Dedekind's $\eta$-function, defined by
 \begin{equation}
 \eta(\tau) := q^\frac{1}{24}\prod_{k\geq 1} \left(1-q^k\right)\label{etadefn}.
\end{equation}
This function  is well known to satisfy the following transformation law \cite{Rad}.
\begin{lemma}\label{ETtrans}
 For
$\gamma=\sm{a}{b}{c}{d} \in \textnormal{SL}_2(\mathbb Z)$, we have that
\begin{equation}
\eta\left(\gamma\tau\right)  = \psi\left(\gamma\right)(c\tau + d)^{\frac{1}{2}} \eta(\tau), \label{etatrgen}
\end{equation}
where $\psi\left(\gamma\right)$ is a $24$th root of unity, which can be given explicitly in terms of Dedekind sums \cite{Rad}.  In particular, we have that
\begin{equation*}
\eta\left(-\frac{1}{\tau} \right)= \sqrt{-i \tau} \eta(\tau).
\end{equation*}
\end{lemma}
We also encounter \emph{almost holmorphic modular forms}, which as originally defined by Kaneko-Zagier \cite{KZ},  transform like usual modular forms,
but are polynomials in $1/v$ with holomorphic coefficients. In this paper we use a slightly modified definition allowing weakly holomorphic coefficients.
 Standard examples of almost holomorphic modular
form include derivatives of holomorphic modular forms, as well as the non-holomorphic Eisenstein series $\widehat{E}_2$, defined by
\begin{align}\label{E2defn}
\widehat{E}_2(\tau) := E_2(\tau) - \frac{3}{\pi v}.
\end{align}
Here its holomorphic part is given by
 \begin{align}\label{E2holdef}E_2(\tau) := 1-24\sum_{n\geq 1} \sigma_1(n) q^n,
 \end{align}
 where $\sigma_1(n)$ is the sum of positive integer divisors of $n$.
 In general, the holomorphic part of an almost holomorphic modular form is called a {\it quasimodular form}.
  After their introduction \cite{KZ}, almost holomorphic modular forms have been shown to play numerous roles in mathematics and physics.  (See for example
 the work of Aganagic-Bouchard-Klemm in \cite{ABK}.)

\subsection{Almost harmonic Maass forms} We are now ready to describe the new automorphic objects occurring in this paper.  For this, we recall  the classical Maass raising operator
 of weight $\kappa$:
 \begin{align}\label{raiselowerdef}
 R_{\kappa}:= 2 i \frac{d}{d \tau} + \frac{\kappa}{v}.
 \end{align}
 If $F$ is of weight $\kappa$ and satisfies $\Delta_\kappa (F) = s F$, one may show that $R_\kappa (F)$ is of weight $\kappa +2$, and that \begin{align*}
\Delta_{\kappa + 2} (R_\kappa (F)) = (s + \kappa) R_\kappa (F).
\end{align*}
 Moreover, define for a positive integer $n$ \begin{equation}\label{eqn:Rrepeat}
R_{\kappa}^n := R_{\kappa+2(n-1)} \circ \cdots R_{\kappa+2} \circ R_{\kappa}.
\end{equation}
Abusing notation, we furthermore set for $F: \mathbb{H}\to \mathbb{C}$
$$
R^{-1}(F)=1.
$$
If   $g$ is modular of weight $\kappa$, then $R_{\kappa}^n (g)$ is an almost holomorphic modular form of weight $\kappa + 2n$.
Motivated by this observation, and in light of the automorphic properties we exhibit for the characters $\textnormal{ch}F_\ell$ (see Theorem \ref{Mainthm} and Theorem \ref{decompth2}), we now define the new objects of interest in this paper.
\begin{definition}\label{nearlyMaass}
Assume the notation above. An \textit{almost harmonic Maass form of weight $\kappa\in\frac12\Z$ and depth
$r \in\mathbb N$ for $\Gamma$ with Nebentypus character $\chi$} is a smooth function $f:\mathbb{H}\rightarrow \C$ satisfying
\begin{enumerate}
\item For all $\gamma=\left(\begin{smallmatrix}
a & b \\ c & d
\end{smallmatrix}\right)\in\Gamma$ and all $\tau\in\mathbb{H}$, we have
$$
f\big|_\kappa \gamma (\tau ) = \chi (d) f(\tau ).
$$
\item The function $f$ may be written as
 $$
f= \sum_{i=0}^r g_i R_{\kappa+2-\nu}^{i-1} (g),
$$
where $g$ is a harmonic weak Maass form of weight $\kappa + 2-\nu$, and the $g_i$ are  almost holomorphic modular forms of weight $\nu-2i$ (for some fixed $\nu \in \frac{1}{2}\mathbb Z$).  \end{enumerate}
\end{definition}
\noindent In particular harmonic weak Maass forms are almost harmonic Maass forms of depth $r=1$, and almost holomorphic modular forms may be viewed as almost harmonic Maass forms of depth $r=0$ or $r=1$. Here, we call the holomorphic parts of almost harmonic  Maass forms \emph{almost mock modular forms}.
We also note that almost harmonic Maass forms inherit the growth from usual Maass forms and almost holomorphic modular  forms.
Moreover, specifying the groups and multipliers in condition (2) would make the transformation law in (1) unnecessary to state.
\subsection{Holomorphic Jacobi forms and mock Jacobi forms} In \cite{ZwegersPhD}, Zwegers studied another type of automorphic object which is commonly referred to as a \emph{mock Jacobi form}.  Before describing these forms, we recall the definition of a holomorphic Jacobi form, after Eichler and Zagier (\cite{EZ} p9).
\begin{definition}\label{harmjacdef} A {\it holomorphic Jacobi form of weight $\kappa$ and index $M$} ($\kappa, M \in \mathbb N$) on a subgroup $\Gamma \subseteq \textnormal{SL}_2(\mathbb Z)$ of finite index is a holomorphic function
$\varphi(z;\tau):\mathbb C \times \mathbb H \to
\mathbb C$ which for all $\gamma = \sm{a}{b}{c}{d} \in \Gamma$ and $\lambda,\mu \in \mathbb Z$ satisfies
 \begin{enumerate}\item
$\varphi\left(\frac{z}{c\tau + d};\gamma\tau\right) = (c\tau + d)^\kappa e^{\frac{2\pi i M c z^2}{c\tau + d}} \varphi(z;\tau)$,
\item $\varphi(z + \lambda \tau + \mu;\tau) = e^{-2\pi i M (\lambda^2\tau + 2\lambda z)} \varphi(z;\tau)$,
\item $\varphi(z;\tau)$ has a Fourier development of the form $\sum_{n, r} c(n,r)q^n e^{2\pi i r z}$ with $c(n,r)=0$ unless $n\geq r^2/4M$.
\end{enumerate}\end{definition}
Jacobi forms with multipliers and of half integral weight, meromorphic Jacobi forms, and weak Jacobi forms are defined similarly with obvious modifications made, and have been studied  in \cite{EZ}, and \cite{ZwegersPhD}, for example.
A  special Jacobi form   used in this paper is Jacobi's theta function, defined by
 \begin{equation}
 \vartheta(z;\tau)= \vartheta(z) :=\sum_{\nu\in\frac12+\Z}e^{\pi i \nu^2\tau+2\pi i\nu\left(z+\frac12\right)}, \label{thedefn}
\end{equation}
where  here and throughout, we may omit the dependence of various functions on the variable $\tau$ when
the context is clear.
This function  is well known to satisfy the following transformation law \cite[(80.31) and (80.8)]{Rad}.
\begin{lemma}\label{THETtrans}
 For $\lambda,\mu \in \mathbb Z$ and
$\gamma=\sm{a}{b}{c}{d} \in \textnormal{SL}_2(\mathbb Z)$, we have that
\begin{eqnarray}
 \vartheta(z+\lambda \tau+\mu;\tau)&=&(-1)^{\lambda+\mu}q^{-\frac{\lambda^2}{2}}e^{-2\pi i\lambda
z}\vartheta(z;\tau),\label{tt1}\\ \vartheta\left(\frac{z}{c\tau+d};
\gamma\tau\right)&=&\psi^3\left(\gamma\right) (c\tau+d)^{\frac12}e^{\frac{\pi icz^2}{c \tau+d}}\vartheta(z;
\tau)\label{tt2}.
\end{eqnarray}
  In particular, we have that
\begin{equation*}
\vartheta\left(\frac{z}{\tau}; -\frac{1}{\tau} \right)= - i \sqrt{-i \tau}  e^{\frac{\pi i z^2}{\tau}}  \vartheta\left(z; \tau \right).
\end{equation*}
\end{lemma}
\noindent The Jacobi theta function is also known to satisfy the well known triple product identity
 \begin{align}\label{JTP}
\vartheta(z;\tau) = -i q^{\frac{1}{8}} \zeta^{-\frac{1}{2}} \prod_{r=1}^\infty (1-q^r)\left(1-\zeta q^{r-1}\right) \left(1-\zeta^{-1}q^r\right),
\end{align} where throughout $\zeta := e^{2\pi i z}$.

Eichler and Zagier have shown
 that holomorphic Jacobi forms have a {\it theta decomposition} in terms of the following
Jacobi theta functions defined for $a\in \mathbb N$ and $b \in \mathbb Z$:
 \begin{align}\label{thetamelldef}
\vartheta_{a, b}(z; \tau) := \sum\limits_{\lambda\in\Z\atop{\lambda\equiv b\pmod{2a}}} e^{\frac{\pi i\lambda^2\tau}{2a}+2\pi
i\lambda z}.
 \end{align}
We summarize results of \cite{EZ} as follows.
 \begin{proposition}[Eichler-Zagier  \cite{EZ}] \label{ezthetadecprop}
Holomorphic Jacobi forms $\varphi(z;\tau)$ of weight $\kappa$ and index $M$ satisfy
$$\varphi(z;\tau) = \sum_{b \pmod {2 M}} h_b(\tau) \vartheta_{M,b}(z;\tau),$$ where $(h_b(\tau))_{b\pmod{2M}}$ is a vector valued modular form of weight $\kappa-\frac12$.
\end{proposition}
\begin{remark}
We note that a similar result is true for congruence subgroups. For ease of exposition we do not state the precise shape here.
\end{remark}
Zwegers mock Jacobi forms don't quite satisfy the elliptic and modular
transformations given in (1) and (2)
in the above definition of holomorphic Jacobi forms, but instead must be \emph{completed} by adding a certain non-holomorphic function in order to satisfy suitable transformation laws.  To describe the simplest case, we define for $z_1, z_2 \in\C\setminus (\Z\tau+\Z)$ and $\tau\in\mathbb{H}$ the function
\[
\mu(z_1, z_2; \tau):=\frac{e^{\pi iz_1}}{\vartheta(z_2; \tau)}\sum_{r\in\Z}
\frac{(-1)^re^{2\pi irz_2}q^{\frac{r^2+r}{2}}}{1-e^{2\pi iz_1}q^r}
\]
and its completion
\[
\widehat{\mu}(z_1, z_2; \tau):=\mu(z_1, z_2)+\frac{i}{2}R(z_1-z_2).
\]
Here the ``nonholomorphic part" is given by
\[
R(z; \tau)=\sum_{r\in\frac12+\Z}\left(\sgn(r)-E\left(\left(r+\frac{\text{Im}(z)}{v}\right)\sqrt{2v}\right)\right)
(-1)^{r-\frac12}q^{-\frac{r^2}{2}}e^{-2\pi i r z}
\]
with
\[
E(z):=2\int_0^z e^{-\pi t^2}dt.
\]
We note that restricting $z_1$ and $z_2$  to torsion points gives a harmonic weak Maass form.
More generally,  the  function $\widehat{\mu}$ satisfies the following transformation laws
\begin{lemma}(Zwegers)\label{muhattranlem}
Assuming the hypothesis as above, we have for $\gamma=\left(\begin{smallmatrix} a & b \\ c & d \end{smallmatrix}\right)\in SL_2(\Z)$ and $r_1,r_2,s_1,s_2 \in \mathbb Z$,
\begin{align*}
\widehat{\mu}\left(\frac{z_1}{c\tau+d}, \frac{z_2}{c\tau+d}; \gamma \tau\right)&=\psi\left(\gamma\right)^{-3}(c\tau+d)^{\frac12}
e^{-\frac{\pi ic(z_1-z_2)^2}{c\tau+d}}\widehat{\mu}(z_1, z_2; \tau), \\
\widehat{\mu}\left(z_1 + r_1\tau + s_1,z_2 + r_2\tau + s_2;\tau\right) &= (-1)^{r_1+s_1+r_2+s_2} e^{\pi i (r_1-r_2)^2 \tau + 2\pi i (r_1-r_2)(z_1-z_2)} \widehat{\mu}(z_1,z_2;\tau),
\end{align*}
where $\psi\left(\gamma\right)$ is as in (\ref{etatrgen}).
\end{lemma}

We further require certain higher index Jacobi forms, but first recall the following ``mock Jacobi forms" of \cite{ZwegersPhD},
defined for $z, w\in \mathbb C$ with $z\neq w$ and $M\in \mathbb N$ by
\begin{align}\label{fudefn} f_w^{(M)}(z;\tau):= \sum_{\alpha \in \mathbb Z} \frac{e^{4\pi iM\alpha z}q^{M\alpha^2}}{1-e^{2\pi i(z-w)}q^\alpha}.\end{align}
The forms $f_w^{(M)}(z;\tau)$ are completed in \cite{ZwegersPhD} as follows:
\begin{align}\label{fmhatdefn}
\widehat{f}^{(M)}_w(z;\tau) := f_w^{(M)}(z;\tau) - \frac{1}{2} \sum_{\ell \pmod{2M}} R_{M,\ell}(w;\tau)
\vartheta_{M,\ell}(z;\tau),\end{align} where
\[
R_{M, \ell}(w; \tau)
:=\sum\limits_{\lambda\in\Z\atop{\lambda\equiv\ell\pmod{2M}}}\left\{\sgn\left(\lambda+\frac12\right)-E\left(\left(\lambda+\frac{2M\textnormal{Im}(w)}{v}\right)\sqrt{\frac{v}{M}}\right)\right\}e^{-\frac{\pi
i\lambda^2\tau}{2M}-2\pi i\lambda w}.
\]
The completed form $\widehat{f}^{(M)}_w(z;\tau)$ transforms like a 2-variable Jacobi form of
weight $1$ and matrix index $\left( \begin{smallmatrix} M&0\\0&-M\end{smallmatrix}\right)$: \begin{lemma}[Zwegers, \cite{ZwegersPhD}] \label{Zwegerstrans}
With notation and hypotheses as above, for all $\gamma=\sm{a}{b}{c}{d}\in\textnormal{SL}_2(\mathbb Z)$ and $\lambda,\mu \in \mathbb Z$, we have that
\begin{align*}
\widehat{f}^{(M)}_{\frac{w}{c\tau+d}} \left(\frac{z}{c\tau+d};\gamma \tau\right)
&= (c\tau+d) e\left(\frac{Mc\left(z^2-w^2\right)}{(c\tau + d)}\right) \widehat{f}^{(M)}_w(z;\tau),\\
\widehat{f}^{(M)}_{w}(z+\lambda \tau + \mu;\tau) &=e^{-2\pi i M(\lambda^2\tau+2\lambda z)}\widehat{f}^{(M)}_{w}(z;\tau), \\
\widehat{f}^{(M)}_{w+\lambda \tau + \mu}(z;\tau) &= e^{2\pi i M(\lambda^2 \tau + 2\lambda w)}\widehat{f}^{(M)}_{w}(z;\tau).
\end{align*}
\end{lemma}
We note that restricting $w$ and $z$ to torsion points gives linear combinations of modular forms multiplied by harmonic weak Maass forms.

Finally, we introduce a useful differential operator $\mathbb D_\varepsilon$, which acts on Jacobi forms \cite{BS}:
\begin{align}\label{Doper}\mathbb D_\varepsilon = \mathbb D_\varepsilon^{(m,n)}(\tau) := \delta_\varepsilon - \frac{(m-n)\textnormal{Im}(\varepsilon)}{\textnormal{Im}(\tau)},
\end{align}
 where $\delta_\varepsilon:=\frac{1}{2\pi i}\frac{d}{d\varepsilon}.$
We point out that the operator $\mathbb D_\varepsilon$ preserves the index of the Jacobi form and raises the weight by $1$.

\section{Automorphic properties of Kac-Wakimoto characters}\label{KWMod}
In this section we prove Theorem \ref{Mainthm}, which follows from Theorem \ref{decompth2}.
\subsection{Preliminaries}
Throughout, we assume that $m>n>0$ with $n \equiv m \equiv 0 \pmod 2$.
It is not difficult to see, using (\ref{JTP}) and the same specializations  as mentioned in the
 introduction,
 that we may rewrite the Kac-Wakimoto generating function (\ref{KWgenf}) as a quotient of Jacobi theta
functions and a power of Dedekind's eta function
\begin{equation} \label{chphiformula}
 chF
 =\varphi\left(z + \frac{\tau}{2};\tau\right)(-1)^mi^{-n}\zeta^{\frac{m-n}{2}}q^{\frac{m-n}{6}}
\eta^{n-m}(\tau),
\end{equation}
where as before $q:=e^{2\pi i\tau}, \ \zeta:=e^{2\pi i z}$ and
\begin{align}\label{phidefn}
\varphi(z; \tau):=\frac{\vartheta\left(z+\frac12\right)^m}{\vartheta(z)^n}.
\end{align}
Given (\ref{chphiformula}) and (\ref{etatrgen}), we proceed by investigating the automorphic properties of the Fourier coefficients of $\varphi(z;\tau)$
(as a function of $z$).
To be more precise, using the notation as in \eqref{helldefn}, studying the Fourier coefficients of $\varphi(z;\tau)$ amounts to studying the functions $h_\ell^{(0)}(\tau)$, and studying the Kac-Wakimoto characters $\textnormal{ch}F_\ell$, i.e. the Fourier coefficients of $\textnormal{ch}F$, amounts to studying the functions $h_\ell^{\left(-\frac{\tau}{2}\right)}(\tau)$.  (Note that as in \cite{DMZ} we have normalized the coefficient functions $h_\ell^{(z_0)}(\tau)$ by multiplying by $q^{-\frac{\ell^2}{2(m-n)}}$.)  Moving from $\varphi(z;\tau)$ to $\textnormal{ch}F$ essentially corresponds to a shift in the Jacobi elliptic variable $\left(z\to z+\frac{\tau}{2}\right)$ of the function $\varphi(z;\tau)$.  As in the setting of holomorphic Jacobi forms, such a shift in $z$ will produce new Jacobi forms;  one only needs to adjust $q$-powers and $\zeta$-powers accordingly, and hence the modular properties of the Fourier coefficents $\textnormal{ch}F_\ell$ can be deduced from those of the coefficient functions $h_\ell^{(0)}(\tau)$. (See \S \ref{mockthetadecsec} for more detail.)
 By a direct calculation we may conclude from Lemma \ref{THETtrans} the following transformation law for $\varphi(z;\tau)$.
\begin{lemma}\label{philem} For $\lambda, \mu \in \mathbb Z$ and $\gamma=\sm{a}{b}{c}{d} \in \Gamma_0(2)$, we have that
\begin{align}\label{phiellip}
\varphi(z+\lambda\tau+\mu)&= q^{-\frac{(m-n)\lambda^2}{2}}e^{-2\pi i(m-n)\lambda z}\varphi(z), \\
\label{phitrans} \varphi\left(\frac{z}{c\tau+d}; \gamma\tau\right)&=\chi^\ast
(\gamma)(c\tau+d)^{\frac{m-n}{2}}e^{\frac{\pi icz^2(m-n)}{c\tau+d}}\varphi(z; \tau),
\end{align}
where
\begin{align}\label{chistardef}
\chi^\ast(\gamma):=\psi\left(\gamma\right)^{3(m-n)}(-1)^{\frac{mc}{4}} .
\end{align}
\end{lemma}
\subsection{Mock theta decompositions}\label{mockthetadecsec}
We now seek a theta-type decomposition for the meromorphic Jacobi form $\varphi$
(see Proposition \ref{ezthetadecprop} for results in the holomorphic setting).  In analogy to their work, Zwegers showed in \cite{ZwegersPhD} that meromorphic Jacobi forms also have a theta-type decomposition, however the components are not modular forms.  Instead, they are shown to be related to mock mock modular forms \cite[Corollary 3.10]{ZwegersPhD}.
 Rather than
proceed according to Zwegers, we take a different approach as first introduced by Dabholkar-Murthy-Zagier \cite{DMZ} in the case
of poles of order $n=1$ and $n=2$.  As explained in \cite{DMZ}, the advantage to their approach for $n=1$ and $n=2$ as opposed
to Zwegers's approach is that  the forms decompose canonically into two pieces, one of which is constructed directly from the
poles, and the other of which is a mock Jacobi form.
Inspired by their treatment, we consider the general Kac-Wakimoto characters, which by way of (\ref{chphiformula}) and Lemma \ref{philem} give rise to Jacobi forms with poles of higher order $n$.
We point out that the situation for more arbitrary $n$ becomes
complicated rather quickly, and extending the approach in \cite{DMZ} from $n=2$ to more arbitrary $n$ is more involved.  In particular, we encounter new automorphic objects which we call almost harmonic Maass forms (see Definition
\ref{nearlyMaass}), that do not appear in \cite{DMZ}.

 As in the treatment in \cite{DMZ} of meromorphic Jacobi forms with poles of order $n=1$ and $n=2$, we let $z_0\in\C$ and define
\begin{align}\label{helldefn}
h_\ell^{(z_0)}(\tau):=q^{-\frac{\ell^2}{2(m-n)}}\int_{z_0}^{z_0+1}\varphi(z; \tau)e(-\ell z)dz.
\end{align}
Here the path of integration is the straight line connecting $z_0$ and $z_0+1$ if there
are no poles on the path of integration.  Otherwise, if there is a pole on the path of integration which is not an endpoint,  we define the path to be the average
of the paths deformed to pass above or below the pole.  If there is a pole at an endpoint, we first remark that the integral defined in  (\ref{helldefn})  depends only on the height of the path of integration, and not on the initial point of the line.  In this case, therefore, we replace the path $[z_0, z_0+1]$ in (\ref{helldefn}) with the path $[z_0-\delta,z_0+1-\delta]$ for some real $\delta>0$ chosen sufficiently small to ensure the endpoints $z_0-\delta$ and $z_0+1-\delta$ are not poles of the integrand.
We note that with this notation, for $0  \leq \ell <m-n$, $ch F_{\ell}$ differs from $h_{\ell}^{(-\frac{\tau}{2})}$ only
by a $q$-power.
In the case of general $\ell$, in moving from $ch F_\ell$ to $h_\ell^{(-\frac{\tau}{2})}$, one picks up further residues which can be made totally explicit, and which yield only a finite number of additional terms in the $q$-expansion.
 Using (\ref{phiellip}), we have that
\begin{equation*}
h_\ell^{(z_0+\tau)}(\tau) = h_{\ell+m-n}^{({z_0})}(\tau),
\end{equation*}
In particular, if we define
\begin{align*}
h_\ell(\tau):=h_\ell^{\left(-\frac{\ell\tau}{m-n}\right)}(\tau),
\end{align*}
we have that
\begin{align}\label{hinvariant}
h_{\ell+m-n}(\tau)=h_\ell(\tau),
\end{align}
where we have used the fact that $n$ is even.
Using the Residue Theorem and the fact that $\varphi$ (as a function of $z$) has poles only in $\Z+ \tau\Z$, we then obtain for $0 \leq \ell <m-n$
$$
h_{\ell}= h_{\ell}^{\left(-\frac{\tau}{2}\right)}.
$$

We next decompose the Jacobi form $\varphi$  into a ``finite part" and a ``polar part" which we define now.    Firstly set
\begin{equation}\label{phiFdefn}
\varphi^F(z; \tau) :=\sum_{\ell\in\Z}h_\ell(\tau)q^{\frac{\ell^2}{2(m-n)}}\zeta^\ell
 =\sum_{\ell\pmod{m-n}}h_\ell(\tau)\vartheta_{\frac{(m-n)}{2}, \ell}\left(z; \tau\right) ,
\end{equation}
where we have used (\ref{hinvariant}).
 The function
$\vartheta_{m, \ell}(z; \tau)$ is as defined in (\ref{thetamelldef}).
  Moreover we  define the ``polar part" of $\varphi$
\begin{equation}\label{phiPdefn}
\varphi^P(z; \tau) :=-\sum_{j=1}^{\frac{n}{2}}\frac{\widetilde{D}_{2j}(\tau)}{(2j-1)!}\delta_\varepsilon^{2j-1} \left[f_\varepsilon^{\left(\frac{m-n}{2}\right)}\left(z;
\tau\right)\right]_{\varepsilon=0},
 \end{equation}
 where the functions $\widetilde{D}_j$ are the Laurent coefficients of   $\varphi(\varepsilon;\tau)$
\begin{equation} \label{quasi}
\varphi(\varepsilon; \tau)=\frac{\widetilde{D}_n(\tau)}{(2\pi i\varepsilon)^n}+\frac{\widetilde{D}_{n-2}(\tau)}{(2\pi
i\varepsilon)^{n-2}}+\ldots+O(1).
\end{equation}
We note that the functions $\widetilde{D}_j$ are quasimodular forms and as such occur as the holomorphic parts of the
  almost holomorphic functions $D_j$. To be more precise, we write
\begin{align}\label{phistarexp}
\varphi^\ast(\varepsilon; \tau)=e^{\frac{\pi(m-n)\varepsilon^2}{2v}}\varphi(\varepsilon; \tau)=\frac{D_n(\tau)}{(2\pi
i\varepsilon)^n}+\frac{D_{n-2}(\tau)}{(2\pi i\varepsilon)^{n-2}}+\ldots+O(1).
\end{align}
The functions   $\widetilde{D}_j$ and the $D_j$ are easily seen to be related  by
\begin{align}\label{Dtotilde}
D_r(\tau)=\sum_{0\leq j\leq \frac{n-r}{2}}(-1)^j\frac{\widetilde{D}_{2j+r}(\tau)}{j!}\left(\frac{m-n}{8\pi v}\right)^j.
\end{align}
Using (\ref{phitrans}) and the fact that
\[
\frac{|c\tau+d|^2}{v}=\left(\frac{(c\tau+d)^2}{v}-2ic(c\tau+d)\right),
\]  we obtain
 \begin{lemma} \label{nearlyhol}
 For $\gamma=\left(\begin{smallmatrix} a & b\\ c & d\end{smallmatrix}\right)\in\Gamma_0(2)$, we have that for $1\leq r\leq n$
\[
D_r\left(\gamma\tau\right)=\chi^\ast\left(\begin{matrix} a & b\\ c &

d\end{matrix}\right)(c\tau+d)^{\frac{m-n}{2}-r}D_r(\tau).
\]
\end{lemma}
We prove the following decomposition theorem
\begin{theorem}\label{decompthm1}
 For the meromorphic function $\varphi(z;\tau)$ defined by (\ref{phidefn}) satisfying
(\ref{phiellip}) and (\ref{phitrans}), we have the decomposition
\begin{align}\label{decomp1}
 \varphi(z;\tau) = \varphi^F(z;\tau) + \varphi^P(z;\tau),\end{align}
where $\varphi^F(z;\tau)$  and $\varphi^P(z;\tau)$ are defined by (\ref{phiFdefn}) and (\ref{phiPdefn}), respectively.
\end{theorem}
\begin{proof}
We  fix a point  ${z_0=A\tau+B\notin \Z+\tau \Z}$.  Since all functions involved in the decomposition theorem  are
meromorphic, it is enough to prove Theorem \ref{decompthm1} for
  $z$ with $\text{Im}(z)=Av$.
  In this case, we have the Fourier expansion
  \[
  \varphi(z; \tau)=\sum_{\ell\in\Z} h_\ell^{(z_0)}(\tau)q^{\frac{\ell^2}{2(m-n)}}\zeta^\ell.
  \]
By definition, we have
\[
\varphi(z; \tau)-\varphi^F(z;
\tau)=\sum_{\ell\in\Z}\Big(h_\ell^{(z_0)}(\tau)-h_\ell(\tau)\Big)q^{\frac{\ell^2}{2(m-n)}}\zeta^\ell.
\]
Using the Residue Theorem, we find that
\[
q^{\frac{\ell^2}{2(m-n)}}\Big(h_\ell^{(z_0)}(\tau)-h_\ell(\tau)\Big)=2\pi i\sum_{\alpha\in\Z}
\frac{\left(\sgn(\alpha-A)-\sgn\left(\alpha+\frac{\ell}{m-n}\right)\right)}{2}\underset{z=\alpha\tau}{\textnormal{Res}}
\Big(\varphi(z; \tau)e(-\ell z)\Big).
\]
We point out that for $\ell \equiv 0 \mod{(m-n)}$, the integral defining $h_\ell(\tau)$ includes poles at endpoints, and we proceed as described in the text following (\ref{helldefn}).
Using the transformation law (\ref{phiellip}) of $\varphi$, we may rewrite
\[
\underset{z=\alpha\tau}{\textnormal{Res}}\Big(\varphi(z; \tau)e(-\ell z)\Big)
=q^{ -\frac{(m-n)\alpha^2}{2}-\ell \alpha}
\underset{\varepsilon=0}{\textnormal{Res}}
\Big(\varphi\left(\varepsilon; \tau\right)e\left(-\varepsilon \left( \ell+(m-n)\alpha \right)
\right)\Big).
\]
Thus
\begin{align*}\nonumber
\varphi(z; \tau)-\varphi^F(z; \tau) =2\pi i\sum_{\alpha\in\Z}&q^{-\frac{(m-n)\alpha^2}{2}}
\frac12\sum_{\ell\in\Z}
\Big(\sgn(\alpha-A)-\sgn\Big(\ell+\alpha(m-n)\Big)\Big)q^{-\ell\alpha}\zeta^\ell\\
&\times\underset{\varepsilon=0}{\textnormal{Res}}\left(\varphi(\varepsilon;
\tau)e\Big(-\varepsilon\Big(\ell+(m-n)\alpha\Big)\Big)\right).
\end{align*}
  We  change $\ell\mapsto\ell-\alpha(m-n)$ and then $\alpha\mapsto -\alpha$, to obtain
\begin{align*}
 \varphi(z; \tau)-\varphi^F(z; \tau) =
-2\pi i\sum_{\alpha\in\Z}&\zeta^{\alpha(m-n)}q^{\frac{(m-n)\alpha^2}{2}} \frac12
\sum_{\ell\in\Z}\Big(\sgn(\alpha+A)+\sgn(\ell)\Big) q^{\ell\alpha}\zeta^{\ell}\nonumber\\& \times
\underset{\varepsilon=0}{\textnormal{Res}}\left(\varphi(\varepsilon; \tau)e(-\varepsilon\ell)\right).
\end{align*}
Using the functions $\widetilde{D}_j$, defined in \eqref{quasi}, it is easily seen that
\begin{equation} \label{computeR}
\underset{\varepsilon=0}{\textnormal{Res}}
\Big(\varphi(\varepsilon; \tau)e(-\ell\varepsilon)\Big)
=\frac{1}{2\pi i}\sum_{j=1}^{\frac{n}{2}}\frac{\widetilde{D}_{2j}(\tau)}{(2j-1)!}\delta_\varepsilon^{2j-1}\Big[e(-\ell\varepsilon)\Big]_{\varepsilon=0}.
\end{equation}
This gives that
\begin{align}\nonumber
\varphi(z; \tau)-\varphi^F(z; \tau)
&=\sum_{j=1}^{\frac{n}{2}}\frac{\widetilde{D}_{2j}(\tau)}{(2j-1)!}\delta_\varepsilon^{2j-1}
 \Bigg[\sum_{\alpha\in\Z} \zeta^{\alpha(m-n)}q^{\frac{(m-n)\alpha^2}{2}}\\\label{phiPcalc3}
&\qquad\frac{-1}{2}\sum_{\ell\in\Z}\Big(\sgn(\alpha+A)+\sgn(\ell)\Big)q^{\ell\alpha}\zeta^\ell
e\big(-\ell\varepsilon\big)\Bigg]_{\varepsilon=0}.
\end{align}
Noting that $\varepsilon$ is real, a direct calculation gives that
\begin{align*}
-\frac{1}{2}\sum_{\ell\in\Z}\Big(\sgn(\alpha+A)+\sgn(\ell)\Big)q^{\ell\alpha}\zeta^\ell
e(-\ell\varepsilon) = -\frac{1}{1-e(z-\varepsilon+\alpha\tau)}+\frac12.
\end{align*}
Thus  the sum on $\alpha$ in (\ref{phiPcalc3}) equals
\begin{equation}\label{alphasum}
-\sum_{\alpha\in\Z}\frac{\zeta^{\alpha(m-n)}q^{\frac{(m-n)\alpha^2}{2}}}{1-e(z-\varepsilon)q^\alpha}+
\frac12 \sum_{\alpha\in\Z}
\zeta^{\alpha(m-n)} q^{\frac{(m-n)\alpha^2}{2}}.
\end{equation}
 By equation \eqref{fudefn}, we find that (\ref{alphasum}) becomes
$$
-f_\varepsilon^{\left(\frac{m-n}{2}\right)}\left(z; \tau\right)+\frac{1}{2}\vartheta_{\frac{m-n}{2}, 0}\left(z; \tau\right).
$$
This finishes the claim by noting that the second summand is independent of $\varepsilon$ and
thus vanishes after differentiating at least once.
\end{proof}

\subsection{Almost harmonic Maass forms}
\noindent
In this section we seek a (nonholomorphic) completion such that the Kac Wakimoto characters transform like automorphic forms.
For this purpose we first rewrite the holomorphic polar part in terms of the nonholomorphic  differential operator $\mathbb D = \mathbb{D}_\varepsilon$ defined in (\ref{Doper}).
\begin{lemma}\label{phiplemma}
We have that
\begin{equation}\label{phiPdefn}
\varphi^P(z; \tau)
 =- \sum\limits_{1\leq j\leq \frac{n}{2}}
 \frac{D_{2j}(\tau)}{(2j-1)!}\mathbb
D_{\varepsilon}^{2j-1}\Bigg[f_\varepsilon^{\left(\frac{m-n}{2}\right)}\left(z; \tau\right)\Bigg]_{\varepsilon=0}.
\end{equation}
\end{lemma}
\begin{proof}
The   proof is similar to that of Theorem \ref{decompthm1}. We first note that (\ref{computeR}) may also be written as
\begin{align} \nonumber
\underset{\varepsilon=0}{\textnormal{Res}}&
\Big(\varphi^*(\varepsilon; \tau) e^{-\frac{\pi (m-n)\varepsilon^2}{2v}}
e(-\ell\varepsilon)\Big)
=\frac{1}{2\pi i}\sum_{j=1}^{\frac{n}{2}}\frac{D_{2j}(\tau)}{(2j-1)}\delta_\varepsilon^{2j-1}\Big[e^{-\frac{\pi(m-n)\varepsilon^2}{2v}}e(-\ell\varepsilon)\Big]_{\varepsilon=0}
\\
&=\frac{1}{2\pi i}\sum_{j=1}^{\frac{n}{2}}\frac{D_{2j}(\tau)}{(2j-1)!}\delta_\varepsilon^{2j-1}\Big[e^{-\frac{\pi(m-n)\varepsilon^2}{2v}}e(-\ell\varepsilon)\Big]_{\varepsilon=0}
= \frac{1}{2\pi i}
\sum_{j=1}^{\frac{n}{2}}\frac{D_{2j}(\tau)}{(2j-1)!}
\mathbb
D_{\varepsilon}^{2j-1}  \Big[
e(-\ell\varepsilon) \Big]_{\varepsilon=0}. \label{phipresdisplay}
\end{align}
One then obtains Lemma \ref{phiplemma} using (\ref{phipresdisplay}) by proceeding as in the proof of Theorem \ref{decompthm1}.
\end{proof}
Next we will define objects that will turn out to transform like true Jacobi forms and (vector-valued) modular forms, respectively. First, we let
 \begin{multline}\label{phiPhatdefn}
 \widehat{\varphi}^P(z;\tau) :=
 \varphi^P(z;\tau) + \frac12 \sum\limits_{1\leq j\leq
\frac{n}{2}}\frac{D_{2j}(\tau)}{(2j-1)!}\sum_{\ell\pmod{(m-n)}}\mathbb{D}_{\varepsilon}^{2j-1}\left[R_{\frac{m-n}{2}, \ell}(\varepsilon;
\tau)\right]_{\varepsilon=0} \vartheta_{\frac{m-n}{2}, \ell}\left(z; \tau\right)\\
=- \sum\limits_{1\leq j\leq
 \frac{n}{2}}\frac{D_{2j}(\tau)}{(2j-1)!}\mathbb D_{\varepsilon}^{2j-1}\Bigg[\widehat{f}_\varepsilon^{\left(\frac{m-n}{2}\right)}\left(z;
 \tau\right)\Bigg]_{\varepsilon=0},
 \end{multline}
 where the forms $\widehat{f}_u^M(z;\tau)$ are defined in
 (\ref{fmhatdefn}).

  Similarly, we look at the completed version of the mock modular form components, and define
 \begin{align}\label{hellhatdef} \widehat{h}_\ell(\tau) := h_\ell(\tau)-
\frac12 \sum\limits_{1\leq j\leq \frac{n}{2}}\frac{D_{2j}(\tau)}{(2j-1)!}\mathbb{D}_{\varepsilon}^{2j-1}\left[R_{\frac{m-n}{2},
\ell}(\varepsilon; \tau)\right]_{\varepsilon=0}
\end{align}
from which we then build a completed version of the finite part of
the function $\varphi(z;\tau)$:
\begin{align}\label{phiFhatdef} \widehat{\varphi}^F(z;\tau) := \sum_{\ell \pmod{(m-n)}} \widehat{h}_\ell(\tau)
\vartheta_{\frac{m-n}{2},\ell}(z;\tau).
\end{align}
Theorem \ref{Mainthm} follows from Theorem \ref{decompth2} below.
 \begin{theorem}\label{decompth2}
The functions $\widehat{\varphi}^P$ and $\widehat{\varphi}^F$ are (nonholomorphic) Jacobi forms of weight
 $\frac{m-n}{2}$  and index $\frac{m-n}{2}$  on $\Gamma_0(2)$ (with some multiplier).
 The functions $\widehat{h}_{\ell}$ are almost harmonic Maass forms of weight $\frac{m-n-1}{2}$ and depth $\frac{n}{2}$.
  \end{theorem}
 \ \\ \emph{Remark.} We note that the multiplier occurring in Theorem  \ref{decompth2} could be made explicit.  (See the proof of Theorem \ref{decompth2}, and \cite{ZwegersPhD}.)
 \begin{proof}
  By Lemma \ref{Zwegerstrans}, Lemma \ref{philem},  Theorem \ref{decompthm1}, and the properties of $\mathbb{D}_{\varepsilon}$ it is not hard to see that $\widehat{\varphi}^P$ transforms like a Jacobi form of weight
 $\frac{m-n}{2}$   and index $\frac{m-n}{2}$ on $\Gamma_0(2)$. As in the case of classical modular forms, one may conclude that
 $\widehat{h}_l$  transforms like a (vector valued) modular form of weight  $\frac{m-n-1}{2}$ and that
 $\widehat{\varphi}^F$ has the same transformation properties as $\widehat{\varphi}^P$.\\
 We next show that
\begin{equation}\label{rewriteDop}
\mathbb{D}_\varepsilon^{2j-1} \left[ R_{\frac{m-n}{2}, \ell} \left( \varepsilon ; \tau \right) \right]_{\varepsilon = 0} = -\frac{(m-n)^{j-1}}{(2\pi)^j} {R^{j-1}_{\frac{3}{2}}} (\mathcal{R} (\tau)),
\end{equation}
where
$$
\mathcal{R} (\tau) := \frac{d}{d\varepsilon} \left[  R_{\frac{m-n}{2}, \ell} \left( \varepsilon ; \tau \right) \right]_{\varepsilon = 0}.
$$
 To see this, we rewrite
  \begin{align}\nonumber
 R_{\frac{m-n}{2}, \ell}(\varepsilon; \tau)
 &=-ie^{-\frac{\pi i\left(\ell-\frac{m-n}{2}\right)^2\tau}{(m-n)}+2\pi i\left(\frac{m-n}{2}-\ell\right)\varepsilon}
\sum_{\lambda\in\frac12+\Z}
\Bigg\{
\sgn\left(\lambda-\frac12+\frac{\ell}{m-n}+\frac{1}{2(m-n)}\right)\\ \nonumber
&-E
\left(\left(\lambda+\frac{\textnormal{Im}\left((m-n)\varepsilon+\left(\ell-\frac{m-n}{2}\right)\tau-\frac12\right)}{(m-n)v}\right)\sqrt{2(m-n)v}\right)
\Bigg\}\\ \nonumber
&\times e^{-\pi i \left(\lambda - \frac12\right)-\pi i(m-n)\lambda^2\tau-2\pi i\lambda\left((m-n)\varepsilon+\left(\ell-\frac{m-n}{2}\right)\tau-\frac12\right)}\\ \label{RtoR}
&=-ie^{-\frac{\pi i\left(\ell-\frac{m-n}{2}\right)^2\tau}{(m-n)}+2\pi i\left(\frac{m-n}{2}-\ell\right)\varepsilon}
 R\left((m-n)\varepsilon- \left(\frac{m-n}{2}-\ell\right)\tau-\frac12;(m-n)\tau\right),
 \end{align}
 where we used that for $0\leq \ell < m-n$, we have
 \[
 \sgn\left(\lambda-\frac12+\frac{\ell}{m-n}+\frac{1}{2(m-n)}\right)=\sgn(\lambda).
 \]
>From work of the first author and Zwegers \cite{BZ}, we may then conclude that
\begin{equation}
\frac{d^2}{d\varepsilon^2} R_{\frac{m-n}{2}, \ell} \left( \varepsilon ; \tau \right) = -4\pi i (m-n) \frac{d}{d\tau} R_{\frac{m-n}{2}, \ell} \left( \varepsilon ; \tau \right).
\end{equation}
This gives that the left hand side of (\ref{rewriteDop}) equals
$$
\frac1{(2\pi i)^{2j-1}}
 \sum_{r=1}^{j} \binom{2j-1}{2r-1} \frac{d^{2j-2r}}{d\varepsilon^{2j-2r}} \left[ e^{\frac{2\pi (m-n) \text{Im} (\varepsilon)^2}{v}} \right]_{\varepsilon = 0}  (-4\pi i (m-n))^{r-1} \frac{d^{r-1}}{d\tau^{r-1}} \mathcal{R}(\tau).
$$
Moreover, using the series expansion of the exponential function, we see that
$$
 \frac{d^{2j-2r}}{d\varepsilon^{2j-2r}} \left[ e^{\frac{2\pi (m-n) \text{Im} (\varepsilon)^2}{v}} \right]_{\varepsilon = 0} = \left( \frac{-i}{2} \right)^{2j-2r} (2j-2r)! \frac{\left(\frac{2\pi (m-n)}{v}\right)^{j-r}}{(j-r)!}.
$$
Thus (\ref{rewriteDop}) equals
$$
-\frac{(m-n)^{j-1}}{(2\pi)^{j}} \sum_{r=1}^{j} \frac{(2j-1)!}{(2r-1)! (j-r)!} i^{r} 2^{3r-2j-1} v^{-j+r} \frac{d^{r-1}}{d\tau^{r-1}} \mathcal{R}(\tau).
$$
Set for $j\geq2$
$$
g_j (\tau) := \sum_{r=1}^j \alpha_{r,j} v^{-j+r}  \frac{d^{r-1}}{d\tau^{r-1}} \mathcal{R}(\tau),
$$
where $$
\alpha_{r,j}:= \frac{(2j-1)!}{(2r-1)!(j-r)!} i^{r} 2^{3r-2j-1}.
$$
To finish the proof, we first note that it is not difficult to see that
$$
g_2 = R_{\frac32} (\mathcal{R}).
$$
We next compute that
$$
R_{2j-\frac12} \left( g_j \right) = \sum_{r=1}^j \frac{\left( r+j-\frac12 \right)}{v^{j+1-r}}\alpha_{r,j} \frac{d^{r-1}}{d\tau^{r-1}} \mathcal{R}(\tau) + 2i \sum_{r=2}^{j+1} \frac{\alpha_{r-1,j}}{v^{j+1-r}} \frac{d^{r-1}}{d\tau^{r-1}} \mathcal{R}(\tau).
$$
From this, one can deduce that
$$
g_{j+1} = R_{2j-\frac12} \left( g_j \right),
$$
{after a short calculation} by comparing this termwise \big(as an expansion in $\frac1v$\big) with
$$
 \sum_{r=1}^{j+1} \frac{\alpha_{r,j+1}}{v^{j+1-r}}  \frac{d^{r-1}}{d\tau^{r-1}} \mathcal{R}(\tau).
$$
This gives (\ref{rewriteDop}).

To finish the proof of the theorem, it remains to show that there exists a holomorphic  $q$-series $f^+$ such that
$$
f(\tau) := f^+ (\tau) + \mathcal{R} (\tau)
$$
is a harmonic Maass form of weight $\frac32$. For this consider the function
\begin{align}\label{fepsiloncomplete}
f(\varepsilon; \tau) := e^{-\frac{\pi i \left( \ell- \frac{m-n}{2} \right)^2\tau}{(m-n)} +2\pi i \left( \frac{m-n}2 - \ell \right) \varepsilon}
\widehat{\mu} \left((m-n) \varepsilon-\frac12, \left( \frac{m-n}{2} - \ell \right) \tau ; (m-n) \tau \right).
\end{align}
We know by Lemma \ref{muhattranlem} that $f$ transforms like a Jacobi form of weight $\frac12$ and index $-\frac{(m-n)}{2}$ (see Lemma \ref{muhattranlem}).
Its nonholomorphic part is given by
$$
\frac{i}{2} e^{-\frac{\pi i \left( \ell - \frac{m-n}{2} \right)^2 \tau}{m-n} + 2\pi i \left( \frac{m-n}2 - \ell \right) \varepsilon} R\left( (m-n) \varepsilon - \left(\frac{m-n}2 - \ell \right) \tau -\frac12; (m-n)\tau\right).
$$
Using (\ref{RtoR}), differentiating $f\left( \varepsilon ; \tau \right)$ with respect to $\varepsilon$ and then setting $\varepsilon = 0$ yields a nonholomorphic modular form of weight $\frac32$ with nonholomorphic part $\frac12 \mathcal{R} (\tau)$. It is not hard to see that this function is annihilated by $\Delta_\frac32$ giving the claim.
\rm
\end{proof}

\section{Examples}\label{examplesec}

\noindent
Here we give explicit examples pertaining to the cases $n=2$ and $n=4$. We compute the almost harmonic Maass forms associated to $\varphi$.
The forms associated to $chF$
are then easily deduced using (\ref{chphiformula}) and (\ref{phidefn}).  We first note that
\begin{equation}\label{Depsilon}
\begin{split}
\mathbb{D}_\varepsilon\left[R_{\frac{m-n}{2}}(\varepsilon)\right]_{\varepsilon=0}
&=\sum_{\lambda\equiv\ell\pmod{m-n}}e^{-\frac{\pi i\lambda^2\tau}{m-n}}\lambda\left(\sgn\left(\lambda+\frac12\right)-E\left(\lambda\sqrt{\frac{2v}{m-n}}\right)\right)\\
&+\frac{\sqrt{m-n}}{\pi\sqrt{2v}}\vartheta_{\frac{m-n}{2}, \ell}(0; -\overline{\tau}).
\end{split}
\end{equation}

\ \\ \ \\ {\bf Example 1: $n=2$.}\\
First we compute the almost holomorphic modular form $D_2$.
 In this case we have that
\[
D_2(\tau)=\widetilde{D}_2(\tau)=\lim_{\varepsilon\to 0}\frac{(2\pi i\varepsilon)^2}{\vartheta(\varepsilon)^2}\vartheta\left(\varepsilon+\frac12\right)^m
=(2\pi i)^2 \frac{\vartheta\left(\frac12\right)^m}{\vartheta'(0)^2}.
\]
It is not difficult to see that
\begin{align*}
\vartheta'(0) &=-2\pi \eta^3(\tau),\\
\vartheta\left(\frac12\right) &=-2\frac{\eta(2\tau)^2}{\eta(\tau)}.
\end{align*}
Thus
\begin{align}\label{D2expr}
D_2(\tau)=-2^{m}\frac{\eta(2\tau)^{2m}}{\eta(\tau)^{m+6}}.
\end{align}
Using (\ref{hellhatdef}) and (\ref{D2expr}), we obtain
\begin{align*} \widehat{h}_{\ell}(\tau) &= h_{\ell}(\tau) + \frac{1}{2} D_2(\tau) \mathbb D_{\varepsilon} \left[R_{\frac{m-n}{2},\ell}(\varepsilon;\tau)\right]_{\varepsilon=0} = h_{\ell}(\tau) - 2^{m-1}\frac{\eta(2\tau)^{2m}}{\eta(\tau)^{m+6}} \mathbb D_{\varepsilon} \left[R_{\frac{m-n}{2},\ell}(\varepsilon;\tau)\right]_{\varepsilon=0}.
\end{align*}
In order to explicitly describe the almost harmonic Maass form in this case, we compute $$\xi_{\frac{3}{2}}\left(\frac{\widehat{h}_\ell(\tau)}{D_2(\tau)}\right).$$
Using (\ref{Depsilon}) (with $n=2$), a simple calculation shows that this is equal to a constant multiple of $\vartheta_{\frac{m-2}{2},\ell}(0;\tau)$.
Summarizing, we have that $\widehat{h}_\ell(\tau)$ is an almost harmonic Maass form of weight $\frac32$ and depth $1$, and the shadow of the associated harmonic weak Maass form is given (up to a constant) by $\vartheta_{\frac{m-2}{2},\ell}(0;\tau)$.   In the case of depth $1$ of course, this statement is equivalent to the statement that
\[
\widehat{h}_\ell(\tau)\frac{\eta(\tau)^{m+6}}{\eta(2\tau)^{2m}}
\]
is a harmonic weak Maass form of weight $\frac32$ with shadow (up to a constant) given by $\vartheta_{\frac{m-2}{2}, \ell}(0; \tau)$.
\ \\ \ \\
{\bf Example 2:  $n=4$.} \ \\   First we compute the almost holomorphic modular forms $D_2$ and $D_4$.  In this case, in order to simplify calculations, we will find $\widetilde{D}_2$ and $\widetilde{D}_4$.  One can then easily make use of (\ref{Dtotilde}) in order to find $D_2$ and $D_4$.  We have that
\begin{align}\label{phiexpan4}\varphi(\varepsilon) := \frac{\vartheta\left(\varepsilon+\frac12\right)^m}{\vartheta(\varepsilon)^4}
=\frac{\widetilde{D}_4(\tau)}{(2\pi i\varepsilon)^4}+\frac{\widetilde{D}_2(\tau)}{(2\pi i\varepsilon)^2}+O(1).\end{align}  \ \\
Let $\vartheta^*(\varepsilon) := \vartheta(\varepsilon)/\varepsilon$. Using the fact that $\vartheta^*(\varepsilon)$ and $\vartheta\left(\varepsilon + \frac{1}{2}\right)$ are even functions of $\varepsilon$, and using the fact that $\vartheta(\varepsilon)$ is holomorphic with simple zeros at $\varepsilon \in \mathbb Z\tau + \mathbb Z$, we have that
\begin{align*}
\vartheta\left(\varepsilon+\frac12\right)&=\vartheta\left(\frac12\right)+\vartheta^{''}\left(\frac12\right)\frac{\varepsilon^2}{2}+O\Big(\varepsilon^4\Big),\\
\vartheta^\ast(\varepsilon) &=\vartheta^\ast(0)+\vartheta^{\ast ''}(0)\frac{\varepsilon^2}{2}+O\Big(\varepsilon^4\Big).
\end{align*}
Thus we rewrite $\varphi(\varepsilon)$ as
\begin{align*}
\varphi(\varepsilon)
&=\frac{\left(\vartheta\left(\frac12\right)+\vartheta^{''}\left(\frac12\right)\frac{\varepsilon^2}{2}+O\Big(\varepsilon^4\Big)\right)^m}
{\varepsilon^4\left(\vartheta^\ast(0)+\vartheta^{\ast ''}(0)\frac{\varepsilon^2}{2}+O\Big(\varepsilon^4\Big)\right)^4} =\frac{1}{\varepsilon^4}\frac{\vartheta\left(\frac12\right)^m}{\vartheta^\ast(0)^4}\frac{\left(1+\frac{m}{2}\varepsilon^2\frac{\vartheta^{''}\left(\frac12\right)}{\vartheta\left(\frac12\right)}+O\Big(\varepsilon^4\Big)\right)}{\left(1+2\varepsilon^2\frac{\vartheta^{\ast ''}(0)}{\vartheta^\ast(0)}+O\Big(\varepsilon^4\Big)\right)}\\
&=\frac{1}{\varepsilon^4}\frac{\vartheta\left(\frac12\right)^m}{\vartheta^\ast(0)^4}\left(1+\varepsilon^2\left(\frac{m}{2}\frac{\vartheta^{''}\left(\frac12\right)}{\vartheta\left(\frac12\right)}
-2\frac{\vartheta^{\ast ''}(0)}{\vartheta^\ast(0)}\right)+O\Big(\varepsilon^4\Big)\right).
\end{align*} One can show, after a short calculation, that  \begin{align}\label{tq1}
\frac{1}{(2\pi i)^2}\frac{\vartheta^{''}\left(\frac12\right)}{\vartheta\left(\frac12\right)}
&=-\frac{1}{12}E_2(\tau)+\frac13 E_2(2\tau),\\\label{tq2}
\frac{1}{(2\pi i)^2}\frac{\vartheta^{\ast ''}(0)}{\vartheta^\ast(0)}&=\frac{1}{12}E_2(\tau),
\end{align}
where $E_2$ is the holomorphic part of the almost modular $\widehat{E}_2$, and is given by (\ref{E2holdef}).  Thus, using (\ref{phiexpan4}), (\ref{tq1}), and (\ref{tq2}) we conclude that
\begin{align}\label{D4tilexplicit}
\widetilde{D}_4(\tau)&=2^m\frac{\eta(2\tau)^{2m}}{\eta(\tau)^{m+12}},\\ \label{D2tilexplicit}
\widetilde{D}_2(\tau)&=2^m\frac{\eta(2\tau)^{2m}}{\eta(\tau)^{m+12}}\left(\left(-\frac{m}{24}-\frac16\right)E_2(\tau)+\frac{m}{6}E_2(2\tau) \right).
\end{align}
Using (\ref{Dtotilde}) we find that
\begin{align}D_4(\tau) &= \widetilde{D}_4(\tau) , \label{Dj4explicit}
\\ \label{Dj2explicit} D_2(\tau) &= \widetilde{D}_2(\tau) - \frac{m-4}{8\pi v}\widetilde{D}_4(\tau),
\end{align}
where $\widetilde{D}_4$ and $\widetilde{D}_2$ are given in  (\ref{D4tilexplicit}) and (\ref{D2tilexplicit}), respectively.

We note that $\mathbb{D}_\varepsilon\left[R_{\frac{m-n}{2}}(\varepsilon)\right]_{\varepsilon=0}$ is given as in \eqref{Depsilon}.
A calculation nearly identical to that in the case $n=2$ in Example 1 above shows that $\xi_{\frac{3}{2}}\left(f(\varepsilon;\tau)\right)$ with $f$ defined in \eqref{fepsiloncomplete} is a multiple of  $\vartheta_{\frac{m-4}{2},\ell}(0;\tau)$.
Thus we have that
$\widehat{h}_\ell$ is an almost harmonic Maass form of weight $\frac{m-5}{2}$ and depth $2$ given by
$$\widehat{h}_\ell(\tau) = \sum_{j=1}^2 D_{2j}(\tau)R^{j-1}_{\frac{3}{2}}\left(g(\tau)\right),$$
where $D_4$ and $D_2$ are the almost holomorphic modular forms given explicitly in (\ref{Dj4explicit}) and (\ref{Dj2explicit}), respectively. The form $g(\tau)$ is a harmonic weak Maass form whose shadow is up to a constant given by $\vartheta_{\frac{m-4}{2},\ell}(0;\tau)$.

\section{The asymptotic behavior of $s\ell(m|n)^\wedge$ Kac-Wakimoto characters}\label{AB}
Here we prove Theorem \ref{Asythm}, which generalizes our previous results in \cite{BF}, which pertain only to the case $n=1$.
We proceed by different methods than in \cite{BF}, and begin by defining the following integral for
$\ell\in \mathbb Z, n, m\in\mathbb N$ with $n<m$
and $t\in \mathbb R^+$:
\begin{equation*}\label{Isints}
 I_{\ell}(t) := \int_\R \frac{e^{\pi(n-m)tu^2-2\pi
i{\ell}ut}}{\cosh(\pi u)^n}du.
\end{equation*}
 To find the Taylor expansion of this integral, we require some more notation.
We define for integers $ k \geq 0$:
 \begin{align}\label{E1defint}
\mathcal{E}_{ k, 1}&:= (-2i)^{-k} E_k,
\ \\ \label{E2defint}
 \mathcal{E}_{k, 2}&:=\frac{2i^{-k}}{\pi}B_k\left(\frac12\right),
\end{align}
where the Euler polynomials $E_k(z)$ and Bernoulli polynomials $B_k(z)$ are given by the generating functions
\begin{align*}
\sum_{k=0}^\infty E_k(z) \frac{x^k}{k!} := \frac{2e^{xz}}{e^x + 1}, \qquad
\sum_{k=0}^\infty B_k\left(z\right) \frac{x^k}{k!}:= \frac{xe^{zx}}{e^x - 1},
\end{align*}
and $E_k := E_k(0)$.
 With the initial functions (\ref{E1defint}) and (\ref{E2defint}),
 we define more general  unique functions $\mathcal E_{k, n+2}$ for $n\geq 1$ and $k\geq 0$ by the following recurrence:
\begin{align}\label{recurEs}
\mathcal{E}_{k, n+2}=\frac{n}{(n+1)}\mathcal{E}_{k, n}+\frac{k(k-1)}{\pi^2n(n+1)}\mathcal{E}_{k-2, n},
\end{align}
where we take $\mathcal{E}_{ k,n}:=0$ if $k<0$.  We require the following Lemma, whose proof we defer to the end of the section
\begin{lemma}\label{intlem}
For $\ell \in \mathbb Z$, $t\in\mathbb R^+$, and any $N \in \mathbb{N}_0$, as $t\to 0^+$ we have that
\[
I_\ell(t)=\sum_{r=0}^N a_r(m, n, \ell)\frac{t^r}{r!}+O\Big(t^{N+1}\Big),
\]
where
\[
a_r(m, n, \ell):=(-\pi)^r\sum_{k=0}^r\left(\begin{matrix} r\\ k\end{matrix}\right)(m-n)^k(2i\ell)^{r-k}\mathcal{E}_{k+r,n}.
\]
\end{lemma}
\ \\ \textbf{Remark.}
We will show in the course of the proof of Lemma \ref{intlem} that the functions given by (\ref{E1d}) and (\ref{E2d}) below explicitly solve the recurrence (\ref{recurEs}).

\begin{proof}[Proof of Theorem \ref{Asythm}]
To prove Theorem \ref{Asythm} we must consider the specialized characters $tr_{L_{m,n}(\Lambda_{(\ell)})}q^{L_0}$ (see (\ref{chtrgenrel}) and \S 4 in \cite{KW2}).
  We let $\tau = it$ and find, using (\ref{etadefn}) and (\ref{chphiformula}), that
  \begin{align} \label{Trphifn}
\textnormal{ch}F \cdot  \prod_{k\geq 1} \left(1-e^{-2\pi kt}\right)
 &=\varphi\left(z +
 \frac{it}{2};it\right)(-1)^mi^{-n}\zeta^{\frac{m-n}{2}}e^{-\frac{\pi t}{3}(m-n) + \frac{\pi t}{12} }\eta^{1+n  -m}( it ) =: \phi(z;t).
\end{align}
Using notation as in $\S \ref{KWMod}$, the expressions given here for the functions $tr_{L_{m,n}(\Lambda_{(\ell)})}q^{L_0}$ depend on the range in which $\zeta$ is taken.  (See the discussion following (\ref{helldefn}).)  Here we are interested in $h_\ell^{\left(-\frac{\tau}{2}\right)}(\tau)$.
By Cauchy's formula, we have that
 \begin{equation} \label{Cauchyformula}
 tr_{L_{m, n}(\Lambda_{(\ell)})}
 q^{L_0}=
\int_{-\frac12}^{\frac12}\phi(u;  t)e^{-2\pi i\ell u}du.
\end{equation}
Note that  no poles of $\phi$ (as a function of $u$) lie   on the chosen path of integration.
To understand the asymptotic behavior of the Kac-Wakimoto characters, we apply modular transformations. Namely, we use Lemma \ref{ETtrans} and Lemma \ref{THETtrans} to obtain
\begin{align}\label{phiinvone}
\phi(z; t ) =\frac{(-1)^{m+n}}{\sqrt{t}}e^{\frac{\pi (n-m)z^2}{t}-\frac{\pi mz}{t}-\frac{\pi ( m-n -1)t}{12}-\frac{\pi m}{4t}}
\frac{\vartheta\left(\frac{z}{it}+\frac12+\frac{1}{2 it }; \frac{ i  }{ t }\right)^m} {\vartheta\left(\frac{z}{ it }+\frac12;
\frac{ i}{ t  }\right)^n}\eta\left(\frac{ i }{ t }\right)^{1+n-m}.
\end{align}
We aim to determine the asymptotic behavior of the right hand side of (\ref{phiinvone}).
First, by using the product expansion (\ref{etadefn}), it is not difficult to see that  as $t\to 0^+$,
we have that
\begin{align}\label{etaitbeh}
\eta\left(\frac{i}{t}\right)=e^{-\frac{\pi}{12 t}}\left(1+O\left(e^{-\frac{2\pi}{t}}\right)\right).
\end{align}
Similarly,  applying  (\ref{JTP}), we  deduce that
\begin{align*}
\vartheta\left(\frac{u}{it}+\frac12+\frac{1}{2it}; \frac{i}{t}\right) &=
-e^{\frac{\pi u}{t}+\frac{\pi}{4t}}\Big(1+O\Big(e^{\frac{2\pi |u|}{t}-\frac{\pi}{t}}\Big)\Big),\\
\vartheta\left(\frac{u}{it}+\frac12; \frac{i}{t}\right) &=-2e^{-\frac{\pi}{4t}}\cosh\left(\frac{\pi
u}{t}\right)
\Big(1+O\Big(e^{\frac{2\pi |u|}{t}-\frac{\pi}{t}}\Big)\Big) .
\end{align*}
Thus, we have that
\begin{align*}
\phi(u; t)
&=
\frac{e^{\frac{\pi(n-m)u^2}{t}-\frac{\pi (m-n-1)t}{12}+{ \frac{\pi (m-1)}{12t}}+\frac{\pi n}{6t}}}{{ \sqrt{t} }2^n\cosh\left(\frac{\pi
u}{t}\right)^n}
 \Big(1+O\Big(e^{\frac{2\pi |u|}{t}-\frac{\pi}{t}}\Big)\Big).
\end{align*}
Inserting this into (\ref{Cauchyformula}) yields \begin{align}\label{trform1}
 {tr}_{L_{m, n}(\Lambda_{(\ell)})}q^{L_0}=\frac{1}{{\sqrt{t}}2^n}e^{  -\frac{\pi (m-n-1) t}{12}+{ \frac{\pi (m-1)}{12t}}+\frac{\pi n}{6t}}
\int_{-\frac12}^{\frac12}\frac{e^{\frac{\pi(n-m)u^2}{t}-2\pi i \ell u}}{\cosh\left(\frac{\pi u}{t}\right)^n}
\Big(1+O\Big(e^{\frac{2\pi |u|}{t}-\frac{\pi}{t}}\Big)\Big)du.
\end{align}
To complete the proof, it thus suffices to understand the asymptotic behavior of
\begin{align}\label{asympint1}
\int_{-\frac12}^{\frac12}\frac{e^{\frac{\pi(n-m)u^2}{t}-2\pi i \ell u}}{\cosh\left(\frac{\pi u}{t}\right)^n}
\Big(1+O\Big(e^{\frac{2\pi |u|}{t}-\frac{\pi}{t}}\Big)\Big)du.
\end{align}
In fact, in the main term we will consider instead the integral in (\ref{asympint1}) over all of $\R$, and first observe that
 \begin{align} \label{intes1}
 \int_{|t|\geq \frac{1}{2}}  \frac{e^{\frac{\pi(n-m)u^2}{t} - 2\pi i \ell u}}{ \cosh\left(\frac{\pi u}{t}\right)^n}du
 \ll \int_{\frac12}^\infty \frac{e^{\frac{\pi(n-m)u^2}{t}}}{ \cosh\left(\frac{\pi u}{t}\right)}du \ll
e^{\frac{\pi(n-m)}{4t}}\int_{\frac12}^\infty e^{-\frac{\pi u}{t}}du \ll te^{\frac{\pi(n-m)}{4t}}.
\end{align}
Next, in order to estimate the error term in (\ref{asympint1}), we bound
  \begin{equation}\label{intes2}
\int_{-\frac{1}{2}}^{\frac{1}{2}} \frac{e^{\frac{\pi(n-m)u^2}{t} + \frac{2\pi|u|}{t}}}{\cosh\left(\frac{\pi u}{t}\right)^{n}}du
=\frac{t}{2} \int_0^{\frac{1}{2t}} \frac{e^{\pi(n-m)tu^2 + 2\pi u}\cos(2\pi \ell u t)}{\cosh(\pi u)^{n}}du
\ll t\int_0^{\frac{1}{2t}}e^{\pi(n-m)tu^2+\pi u}du.
\end{equation}
It is easy to see that the integrand achieves a maximum at $u=(2(m-n)t)^{-1},$ bounding (\ref{intes2}) (up to a constant)
by $\exp\big(\pi / (4t(m-n))\big)/2$.
Taking  the extra $\exp(-\pi/t)$ in the error term  of (\ref{asympint1}) into
account, we find that this term only contributes to the overall error.

To finish the proof, we consider the Taylor expansion of
\[
\int_\R \frac{e^{\frac{\pi(n-m)u^2}{t}-2\pi i\ell u}}{\cosh\left(\frac{\pi u}{t}\right)^n}du,
\]
which is easily seen to be equal to $tI_\ell(t)$.
Using Lemma \ref{intlem} with (\ref{trform1}) and the preceding argument gives
\[
{tr}_{L_{m, n}(\Lambda_{(\ell)})}q^{L_0}=e^{\frac{\pi t}{12}(n-m+1)+\frac{\pi}{12 t}(m+2n-1)}\frac{\sqrt{t}}{2^n} \left(\sum_{r=0}^Na_r(m, n,
\ell)\frac{t^r}{ r!}+O\left(t^{N+1}\right)\right)
\]
as claimed in Theorem \ref{Asythm}.
\end{proof}

  The remainder of this section is devoted the proof of Lemma \ref{intlem}.
\begin{proof}[Proof of Lemma \ref{intlem}]
We compute that
\begin{align}\nonumber
 \frac{d^r}{dt^r}  \Big[I_\ell(t)\Big]_{t=0}&=(-\pi)^r\int_\R\frac{\Big((m-n)u^2+2
i\ell u\Big)^r}{\cosh(\pi u)^n}du \\ &=(-\pi)^r\sum_{k=0}^r\left(\begin{matrix} r\\  k\end{matrix}\right)(m-n)^k
(2i \ell )^{r- k}\int_\R\frac{u^{k+r}}{\cosh(\pi u)^n}du.\label{Ideriv}
\end{align}
In light of the integral appearing in (\ref{Ideriv}), we define
\begin{align}\label{etildes}
\widetilde{\mathcal{E}}_{k, n} := \int_{\mathbb R}\frac{x^k}{\cosh(\pi x)^n}dx.
\end{align}
To finish the proof, we claim that
the functions $\widetilde{\mathcal E}_{k,n}$  satisfy the recurrence (\ref{recurEs}), with the initial functions given by (\ref{E1d}) and (\ref{E2d}).
We first establish the following lemma, which shows that $\mathcal{E}_{k,n} = \widetilde{\mathcal{E}}_{k,n}$ for $n=1,2$.
\begin{lemma}\label{eulerlemma}For non-negative integers $k$, we have that
 \begin{align}\label{E1d}
 (-2i)^{-k} E_k  &=\int_\R\frac{x^k}{\cosh(\pi x)}dx,
\ \\ \label{E2d}
\frac{2i^{-k}}{\pi}B_k\left(\frac12\right)&=\int_\R\frac{x^k}{\cosh(\pi x)^2}dx.
\end{align}
\end{lemma}
\begin{proof}[Proof of Lemma \ref{eulerlemma}]
Identity (\ref{E1d}) is known and can be found in \cite{EMOT}, p $\leftarrow$ 42, (18),  for example.  Since we did not find a proof of (\ref{E2d}) we provide one here for the convenience of the reader.
We first claim that for real $z\neq 0$
 \begin{equation} \label{integraliden}
  \int_\mathbb R \frac{e^{2\pi i z x}}{\cosh(\pi x)^2} dx = \frac{4z}{e^{\pi z} -e^{-\pi z}}
  \end{equation}
To see that (\ref{integraliden}) is true, set  $f(z,x):= \displaystyle \frac{e^{2\pi i z x}}{\cosh(\pi x)^2}.$
  Using the   Residue Theorem we obtain, noting that   $x=i/2$ is a double pole of $f(z,x)$  and that
  $f(z, \pm R+it)\to 0$ for $R \to \infty$, that
  \begin{equation} \label{RTheorem}
  \left(\int_{\mathbb R} - \int_{\mathbb R+i}\right) f(z,x) \ dx = 2\pi i {\textnormal{ \ Res}}_{x= \frac{i}{2}} f(z,x)
  = 4 ze^{-\pi z}.
  \end{equation}
   Next we observe that
     \begin{equation} \label{doubleup}
     \int_{\mathbb R + i } f(z,x) \ dx= \int_{\mathbb R} f(z,x+i) \ dx = e^{-2\pi z}\int_{\mathbb R} f(z,x) \ dx.
     \end{equation}
Combining  (\ref{RTheorem}) and (\ref{doubleup}), we obtain (\ref{integraliden}).

Identity (\ref{integraliden}) then implies that
\begin{equation} \label{ps1}
\int_{\mathbb R} \frac{e^{itx}}{\cosh(\pi t)^2} dt =\frac{2xe^{\frac{x}{2}}}{\pi(e^{x} - 1)}
= \frac{2}{\pi} \sum_{j=0}^{\infty} B_j\left(\frac12\right) \frac{x^j}{j!}.
\end{equation}
On the other hand,
\begin{align}
\int_{\mathbb R} \frac{e^{itx}}{\cosh(\pi t)^2}  dt= \sum_{j=0}^{\infty} i^j
 \widetilde{\mathcal{E}}_{j, 2}\frac{x^j}{j!}
 .\label{ps2}
 \end{align}
Comparing power series coefficients in (\ref{ps1}) and (\ref{ps2}), we obtain identity (\ref{E2d}). \end{proof}

Resuming the proof of Lemma \ref{intlem}, we begin with the fact that $\mathcal E_{k,n} = \widetilde{\mathcal{E}}_{k,n},$ for $n=1,2$ as established by Lemma \ref{eulerlemma}. Next, we use the identity
\[
\left(\frac{d}{dx}\right)^2\left(\frac{1}{\cosh(\pi x)^n}\right)=\frac{\pi^2n^2}{\cosh(\pi x)^n}-\frac{\pi^2n(n+1)}{\cosh(\pi
x)^{n+2}},
\]
and proceed by integration by parts.   We find that
\begin{align}\nonumber
\widetilde{\mathcal{E}}_{k, n+2} &=\int_\R\frac{x^k}{\cosh(\pi x)^{n+2}}dx \\ \nonumber
&=\frac{n}{(n+1)}\int_\R\frac{x^k}{\cosh(\pi
x)^n}dx-\frac{1}{\pi^2 n(n+1)}\int_\R x^k\left(\frac{d}{dx}\right)^2\left(\frac{1}{\cosh(\pi x)^n}\right)dx\\
&=\frac{n}{(n+1)}\widetilde{\mathcal{E}}_{k, n}+\frac{k(k-1)}{\pi^2 n(n+1)}\widetilde{\mathcal{E}}_{k-2, n}\label{tilderec}
\end{align}
as claimed.  Considering the Taylor expansion about $t=0$ of $I_\ell(t)$, the proof of Lemma \ref{intlem} now follows from (\ref{Ideriv}), (\ref{etildes}), and (\ref{tilderec}).
\end{proof}

\end{document}